\documentclass[12pt,a4paper]{amsart}
\usepackage{amsfonts,color}
\usepackage{amsthm}
\usepackage{amsmath}
\usepackage{amscd}
\usepackage[utf8]{inputenc}
\usepackage{t1enc}
\usepackage[mathscr]{eucal}
\usepackage{indentfirst}
\usepackage{url}
\usepackage{thmtools}
\usepackage{comment}
\usepackage{floatflt }
\usepackage{ amssymb }

\usepackage{hyperref}

\numberwithin{equation}{section}
\usepackage[margin=2.6cm]{geometry}

\theoremstyle{plain}
\newtheorem{Th}{Theorem}[section]
\newtheorem{Lemma}[Th]{Lemma}

\theoremstyle{definition}
\newtheorem{Def}[Th]{Definition}
\newtheorem{Conj}[Th]{Conjecture}

\renewcommand{\S}{{\mathcal{S}}}
\newcommand{\m}{\mathcal}
\newcommand{\trivup}{\left\lfloor\frac{n}{k}\right\rfloor}

\newcommand{\NN}{n}
\def\vec#1{\widehat{#1}}
\def\silent#1\par{}
\def\bbF{{\mathbf F}}
\def\cB{{\mathcal{B}}}
\def\cA{{\mathcal{A}}}
\def\cL{{\mathcal{L}}}
\def\cP{{\mathcal{P}}}
\def\cS{{\mathcal{S}}}

\makeatletter
\newcommand{\mainsectionstyle}{%
  \renewcommand{\@secnumfont}{\bfseries}
  \renewcommand\section{\@startsection{section}{1}%
	\z@{.7\linespacing\@plus\linespacing}{.5\linespacing}%
	{\normalfont\large\scshape\centering\bfseries}}
}
\makeatother

\makeatletter
\def\paragraph{\@startsection{paragraph}{4}%
  \z@\z@{-\fontdimen2\font}%
  {\normalfont\bfseries}}
\makeatother

\usepackage{tikz}
\usetikzlibrary{svg.path}

\newcommand*{\svgperiod}{%
  \leavevmode
  \tikz[baseline=0pt, x=1pt, y=1pt, scale=1em/1000]\fill
    svg {
      M192 53c0 -29 -24 -53 -53 -53s-53 24 -53 53s24 53 53 53s53 -24 53 -53z
    }
    (current bounding box.west) ++(-86, 0) 
    (current bounding box.east) ++(85, 0) 
  ;%
}

\begin{document}
\mainsectionstyle
\title{Clique number of xor-powers of Kneser graphs}
\author{
Zoltán Füredi
\and
András Imolay
\and
Ádám Schweitzer}
\thanks{
Research
of the first author was supported in part by the National Research Development and Innovation Office, NKFIH,
grants K--132696 and KKP 133819.}



\begin{abstract}
Let $f_\ell(\NN, k)$ denote the clique number of the xor-product of $\ell$ isomorphic Kneser graphs $KG(\NN,k)$.
Alon and Lubetzky investigated the case of complete graphs as a coding theory problem and showed $f_\ell(\NN,1)\leq \ell \NN +1$.
Imolay, Kocsis, and Schweitzer proved that $f_2(\NN,k)\leq \trivup+c(k)$.

Here, the order of magnitude of $c(k)$ is determined to be $\Theta\left( k \binom{2k}{k}\right)$.
By explicit constructions and by an algebraic proof, it is shown that  $\ell\NN- 2\ell-1 \leq f_\ell(\NN,1)\leq \ell\NN-\ell+1$ (for all $\NN\geq 1$ and $\ell\geq 3$).
Finally, 
 it is proved that the order of magnitude of $f$ lies between $\Omega\left(\NN^{\left\lfloor \log_2(\ell+1)\right\rfloor}\right)$ and $O\left(\NN^{\left\lfloor \frac{\ell+1}{2}  \right\rfloor} \right)$  (as $\ell$, $k$ are given and  $\NN\to \infty$).

We conjecture that the lower bound gives the correct exponent.
\end{abstract}

\maketitle

\section{Introduction}

\subsection{Kneser graphs}
A Kneser graph $G:= KG(A,k)$ has a {\em base} set $A$, the vertex set of $G$ consists of all subsets of $k$ elements of $A$. We denote this as $V(G):= \binom{A}{k}$, and a pair $\{X,Y\}$ forms an edge of $G$ when $X\cap Y=\emptyset$. We also use $KG(n,k)$ for a Kneser graph with an $n$-element base set.
A complete subgraph in the Kneser graph corresponds to a family of mutually disjoint $k$-element sets in the base set $A$. So, the size of the largest clique
 $\omega(KG(\NN,k))=\lfloor \NN/k\rfloor$.

The parameters of Kneser graphs are widely studied in combinatorics. Lovász~\cite{lovasz1978kneser} determined the chromatic number of Kneser graphs and Erdős, Ko and Rado~\cite{zbMATH03162924} determined their independence number.
Brešar and Valencia-Pabon~\cite{BRESAR20191017} examined the independence number of Kneser graphs of different graph products.
In this article, we study the clique number of the xor-products.

\subsection{The xor-product}
Given two graphs $G=(V(G), E(G))$ and $H=(V(H), E(H))$, their {\em xor-product} $G \cdot H$ is a graph with the vertex set $V(G) \times V(H)$ and two vertices $(g,h)$ and $(g',h')$ are connected in $G \cdot H$ if and only if either $gg' \in E(G)$ and $hh' \not\in E(H)$ or $gg' \not\in E(G)$ and $hh' \in E(H)$.
The xor-product is not as well understood as other graph products, but there are a number of highly non-trivial results about it, e.g., by Alon and Lubetzky~\cite{xorproduct} and Thomason~\cite{Thomason1997GraphPA}.
They were also motivated to compare it to the Shannon capacity of graphs, see Alon and Lubetzky~\cite{AL2}, Lov\'asz~\cite{Lovasz_Shannon}.
Let $f_\ell(\NN, k)$ denote the clique number of the xor-product of $\ell$ isomorphic Kneser graphs $KG(\NN,k)$.

Taking a clique $C\subset V(G)$ and a vertex $b\in V(H)$ the set $C\times \{ b\}$ forms a clique in $G\cdot H$ so we obtain $\omega (G\cdot H)\geq \max \{ \omega(G), \omega(H)\}$.
Hence
\begin{equation*}\label{eq111}
\omega\left(KG(\NN,k)\cdot KG(\NN,k) \right) \geq \left\lfloor \NN/k \right\rfloor.
 \end{equation*}
Imolay, Kocsis, and Schweitzer~\cite{imolay2021clique} showed that the function $f_2(\NN,k)- \left\lfloor \NN/k \right\rfloor$ is bounded
 for any given $k$.
Define
$$ c(k):=  \sup_{\NN\to \infty} \left\{f_2(\NN,k)-\trivup\right\}.$$
One of the objectives of this article is to determine the order of magnitude of $c(k)$ as $k \to \infty$.

\begin{Th} \label{maintheorem}
For all $k \geq 1$ and $\NN\geq \frac{1}{2}\left(\binom{2k}{k}-2\right) k^2$,
\begin{equation}\label{eq11}
  f_2(\NN,k) \geq \trivup+\binom{2k}{k}\frac{k}{2}-k.
\end{equation}
On the other hand, as $k\to \infty$ we have
\begin{equation}\label{eq12}
  c(k) \leq (1+o(1))k\binom{2k}{k}.
\end{equation}
\end{Th}

The proof of Theorem~\ref{maintheorem} is presented in Section~\ref{twosets}.
It might be easier to determine $f_2(\NN,k)$ for large $\NN$.
In~\cite{imolay2021clique} it was proved that $f_2(\NN,2)=\lfloor{\NN/2}\rfloor +4$ for sufficiently large $\NN$.
Let us define
  $$c_\infty(k):=\limsup_{\NN\to \infty} \left\{f(\NN,k)-\trivup\right\}.$$
We have $c_\infty(1)=c(1)=0$, $c_\infty(2)=4$, and in general the true order of magnitudes
\begin{equation*}\label{eq122}
    \binom{2k}{k}\frac{k}{2}-k \leq c_\infty(k) \leq c(k)\leq (1+o(1))k\binom{2k}{k}.
\end{equation*}
We conjecture that here the equality holds for $ c_\infty(k)$ for all $k\geq 1$ and the best construction is the one from Section~\ref{ss_lowerboundconstruction}.

\begin{restatable}
{Conj}{fosejtes} \label{conj_exact}
For all $k$, if $\NN$ is large enough then
$$f_2(\NN,k)=\trivup+\binom{2k}{k}\frac{k}{2}-k. $$
\end{restatable}

Maybe more is true and $c_\infty(k) = c(k)$ for all $k$.

\subsection{Multiproducts of the complete graphs} \label{ss_13}


The rest of our article tackles the question of higher powers of Kneser graphs. We investigate $f_\ell(\NN,k)$ the clique number of the $\ell$-th xor-power (the xor-product of $\ell$ isomorphic copies) of the Kneser graph $KG(\NN,k)$.

Even the case $k=1$ is not trivial when $\ell  \geq 3$.
Note that $KG(\NN,1)$ is the complete graph on $\NN$ vertices.
The function $f_\ell(\NN,1)$ was considered by Alon and Lubetzky in~\cite{xorproduct}, they proved an upper bound $f_\ell (\NN,1)\leq \ell \NN+1$.
Here we give tighter bounds.

\begin{Th}\label{th_1}
For all $\ell \geq 3$ and $\NN \geq 1$
$$\ell\NN- 2\ell-1 \leq f_\ell(\NN,1)\leq \ell\NN-\ell+1.$$
\end{Th}

\subsection{Higher powers of Kneser graphs} \label{ss_14}

We give bounds for the magnitude of $f_\ell(\NN,k)$ for large $\NN$. In particular, we show that it is not necessarily  linear in $\NN$.

\begin{Th} \label{thm_generalpowers}
We have
\begin{equation}\label{eq13}
 f_\ell(\NN,k)\leq  2^{\left\lfloor \frac{\ell}{2} \right\rfloor}\cdot  \left\lfloor \frac{\ell}{2} \right\rfloor !  \cdot \NN^{\left\lfloor \frac{\ell+1}{2} \right\rfloor} .
 \end{equation}
On the other hand, if $k \geq \left\lfloor \log_2(\ell+1)\right\rfloor$, then
 \begin{equation}\label{eq14}
    f_\ell(\NN,k) \geq \left(\left\lfloor \frac{\NN}{k} \right\rfloor \right)^{\left\lfloor \log_2(\ell+1)\right\rfloor} .
    \end{equation}
\end{Th}

This settles the exact magnitude for the cases $\ell\leq 4$.

\begin{Conj}
For any fixed $\ell$ and $k$ if $k$ is  large enough then
    $$f_\ell(\NN,k)=\Theta(n^{\left\lfloor \log_2(\ell+1)\right\rfloor}).$$
\end{Conj}

\subsection{Semi-intersecting families}

The vertex set of a Kneser graph $KG(A,k)$ is a $k$-uniform hypergraph.
Given $\ell$ Kneser graphs $KG(A_i,k)$ $(1\leq i\leq \ell)$ with pairwise disjoint $n$-element base sets $A_1, \dots, A_\ell$
 the vertices of their xor-product 
 are naturally correspond to those $k\ell$-element subsets $S$ of
 $A_1 \cup A_2 \cup \ldots \cup A_\ell $ where 
 $|S\cap A_i|=k$ for each $i$.
The set pair $\{ S, S'\}$ corresponds to an edge in the xor-product $KG(A_1,k)\cdots KG(A_\ell,k)$ if $S\cap S'\cap A_i=\emptyset$ in an odd number of cases for $1\leq i\leq \ell$.

\begin{Def}
A family of sets $\mathcal{S}$ on the pairwise disjoint base sets $A_1 \cup A_2 \cup \ldots \cup A_\ell $ is called an $\ell ${\em -semi-intersecting family} with parameters $\NN$ and $k$ if 
\begin{itemize}
    \item $|A_1|=|A_2| =\dots =|A_\ell|=n$,
    \item $|S\cap A_1|=|S\cap A_2|=\ldots=|S\cap A_\ell |=k$ for each $S \in \mathcal{S}$, and
    \item for distinct $S$, $T \in \m{S}$, we have $S\cap T\cap A_i=\emptyset$ for an odd number of $i$'s, $1 \leq i \leq \ell$.
\end{itemize}
\end{Def}

There is a one-to-one correspondence between $\ell$-semi-intersecting families and cliques in $KG(\NN,k)^\ell$. Hence
 $f_\ell (\NN,k)= \omega \left(KG(\NN,k)^\ell\right)$ is the maximum size of an  $\ell $-semi-intersecting family with parameters $\NN$ and $k$.

We prefer to work with this equivalent hypergraph reformulation. 
Similar questions in extremal combinatorics with two part set systems were studied extensively, see, e.g., \cite{gerbner2012almost} and~\cite{article}.

\section{Determining the order of magnitude of $c(k)$} \label{twosets}

In this section we prove the bounds stated in Theorem~\ref{maintheorem}.
In this case $\ell=2$, we use simply semi-intersecting instead of 2-semi-intersecting, and we denote the base sets $A_1$, $A_2$ by $A$ and $B$.

\subsection{Lower bound construction} \label{ss_lowerboundconstruction} Here we give a construction yielding~\eqref{eq11}.
\begin{proof}
For easier notation introduce $m:=\frac{1}{2}\binom{2k}{k}$. Choose a subset of $K \subset A$ with $|K|=2k$. Label the $k$-element subsets of $K$ by
  $$H_1, H_2, \ldots, H_{m}, G_1, G_2, \ldots, G_{m}$$ such that $H_i$ and $G_i$ are disjoint ($1 \leq i \leq m$).
Let $L_2, L_3, \ldots,  L_m$ be pairwise disjoint $k^2$-element subsets of $B$. This is possible as $\NN\geq \frac{1}{2}\left(\binom{2k}{k}-2\right) k^2$. Let us arrange the elements of $L_i$ to a $k\times k$ rectangular point lattice.
There are $n-(m-1)k^2$ elements of $B\setminus \bigcup L_i$, so we can select additional pairwise disjoint $k$-element subsets $F_1,F_2,\ldots,F_d$ from them, where $d=\left\lfloor\frac{\NN}{k}\right\rfloor-(m-1)k.$

Define a semi-intersecting family $\m{S}$ as follows. We let $S\subset A \cup B$ in $\m{S}$ if one of the following holds.
\begin{itemize}
    \item $S \cap A=H_i$ and $S \cap B$ corresponds to a row of the lattice in $L_i$, 
    \item $S \cap A=G_i$ and $S \cap B$ corresponds to a column of the lattice in $L_i$, 
    \item $S \cap A=H_1$ and $S \cap B=F_i$ for some $1 \leq i \leq d$.
\end{itemize}

This $\m{S}$ is a semi-intersecting family with parameters $\NN$ and $k$, because if $S,T \in \m{S}$ with $S \neq T$, and they are disjoint in $A$ then $\{S \cap A, T \cap A\}=\{H_i, G_i\}$ for some $2 \leq i \leq m$. On the other hand, this is the only case when $S$ and $T$ intersect in $B$, as they intersect only if $S \cap B$ is a row (or column) of some $L_i$ and $T \cap B$ is a column (or row) of the same $L_i$.

It remains to count the members of $\m{S}$. There are $(m-1)k$ sets from both the first and second bullet point, and $d$ from the third one.
Hence $$|\m{S}|=2(m-1)k+d=\left(\binom{2k}{k}-2\right)k+\trivup - \frac{\binom{2k}{k}-2}{2}k
=\left\lfloor\frac{n}{k}\right\rfloor+\binom{2k}{k}\frac{k}{2}-k.$$
\end{proof}

\subsection{Cross intersecting matchings}\label{ss_cross}

A set of hypergraphs $\cA_1, \dots, \cA_t$ is called a $k$-uniform {\em cross intersecting} matching of size $t$ and of type $(d_1, \dots,  d_t)$
 if $t,k\geq 2$, $d_i \geq 2$ for $1 \leq i \leq t$, each $\cA_i$ consists of $d_i$ pairwise disjoint $k$-element sets and $X\cap Y\neq \emptyset$ for
 $X\in \cA_i$, $Y\in \cA_{j}$ whenever $1\leq i\neq j\leq t$. Obviously, $d_i\leq k$ for every $i$.
The classical set-pair theorem of Bollob\'as~\cite{bollobas1965generalized} implies that $t\leq \frac{1}{2}\binom{2k}{k}$ and here equality holds only if $d_1=\dots =d_t=2$ and each $\cA_i$ consists of a complementary pair of $k$-sets of a base set $L$, $|L|=2k$.
There are many generalizations of Bollob\'as's theorem, see, e.g., Alon~\cite{Alon85} where the exterior algebra method was introduced. Even the case of 2-independent $d$-partitions is highly non-trivial, i.e., when $\cup \cA_i$ is the same $dk$-element set for all $i$. For this case Gargano, K\"orner, and Vaccaro~\cite{KornerJ} showed that for any fixed $d$ one can have $t= \Omega( 4^{k(1-o(1))})$
using their Sperner capacity method in information theory.
Here we show an upper bound.

\begin{Lemma}\label{cross-int}
There exists a sequence $\gamma(2), \gamma(3), \ldots$ with $\gamma(k) \to 0$ exponentially as $k \to \infty$, such that for every $k$-uniform cross intersecting matching of type $(d_1, \dots, d_t)$,
\begin{equation}\label{eq221}
  \sum_i d_i(d_i-1) \leq (1+ \gamma(k))\binom{2k}{k}.
\end{equation}
\end{Lemma}

We {\em conjecture} that the true value of $\gamma$ is zero, for all $k$.

\begin{proof}
Define $A:= \cup_i (\cup \cA_i)$, $\NN:= |A|$.
Given any ordering $\pi$ of $A$ and two non-empty subsets $X,X' \subset A$, we say that $X<_\pi X'$ if $\pi(x)<\pi(x')$ for all $x \in X$ and $x' \in X'$.
We call $\pi$ of {\em type} $i$ if there are two sets $X, X' \in \mathcal{A}_i$ with $X <_\pi X'$.
Every permutation $\pi$ can have only at most one type. Indeed,
if $X_i, X'_i \in \m{A}_i$ with $X_i <_\pi X'_i$ and $X_j, X'_j \in \m{A}_j$ with $X_j<_\pi X'_j$ then the cross intersection property implies that there are  elements $u \in X_i \cap X'_j$ and $v \in X'_i \cap X_j$. From $X_i<_\pi X'_i$ we get $\pi(u)<\pi(v)$ and from $X_j<_\pi X'_j$ we get $\pi(v)<\pi(u)$, a contradiction.

Consider a uniform probability distribution on the $n!$ possible orderings of $A$.
Let $E_i$ denote the event that the random variable $\pi$ is of type $i$.
We have $\sum_{i} \Pr(E_i)\leq 1$ because for $i \neq j$ the events $E_i$ and $E_j$ cannot occur simultaneously.
Fix $i$, our goal is to approximate $\Pr(E_i)$.
From now on, $X_1,X_2,\ldots, X_{d_i}$ denotes the members of $\mathcal{A}_i$.
To simplify the presentation we leave out the index $i$ from $d_i$ in the following calculation.


Let $O_{\alpha,\beta}$ be the event that $X_\alpha<X_\beta $. There are $d(d-1)$ such events, because $1 \leq \alpha,\beta \leq d$ and $\alpha \neq \beta$. Since $E_i=\bigcup_{\alpha\not=\beta}O_{\alpha,\beta}$ the inclusion-exclusion principle yields
$$\Pr(E_i)\geq\sum_{\alpha\not=\beta}\Pr(O_{\alpha,\beta})-\frac{1}{2}\sum_{\alpha_1\not=\beta_1,\alpha_2\not=\beta_2,(\alpha_1,\beta_1)\not=(\alpha_2,\beta_2)}\Pr(O_{\alpha_1,\beta_1}\cap O_{\alpha_2,\beta_2}).$$

In the first sum, $\Pr(O_{\alpha,\beta})=\frac{1}{\binom{2k}{k}}$, because this is the probability that $\pi$ arranges the elements of $X_\alpha\cup X_\beta$ such that $X_\alpha< X_\beta$.
Now we calculate $\Pr(O_{\alpha_1,\beta_1}\cap O_{\alpha_2,\beta_2})$ for all possible $\alpha_1, \beta_1, \alpha_2, \beta_2$. We categorize them  into 
 six groups.


---\enskip
$\alpha_1=\beta_2$ and $\alpha_2=\beta_1$. In this case
$$\Pr(O_{\alpha_1,\beta_1}\cap O_{\alpha_2,\beta_2})=0.$$

---\enskip
 The numbers $\alpha_1, \beta_1, \alpha_2$ and $\beta_2$ are all distinct. In this case $O_{\alpha_1,\beta_1}$ and $O_{\alpha_2,\beta_2}$ are independent events, hence
$$\Pr(O_{\alpha_1,\beta_1}\cap O_{\alpha_2,\beta_2})=\frac{1}{\binom{2k}{k}\binom{2k}{k}}.$$
Here there are  $d(d-1)(d-2)(d-3)$ possibilities for $\alpha_1, \beta_1, \alpha_2, \beta_2$.

---\enskip
 $|\{\alpha_1, \beta_1, \alpha_2, \beta_2\}|=3$ and $\alpha_2=\beta_1$. In this case $X_{\alpha_1}<X_{\beta_1}=X_{\alpha_2}<X_{\beta_2}$, therefore
$$\Pr(O_{\alpha_1,\beta_1}\cap O_{\alpha_2,\beta_2})=\frac{1}{\binom{3k}{2k}\binom{2k}{k}}.$$
There are $d(d-1)(d-2)$ such possibilities.

---\enskip
 $|\{\alpha_1, \beta_1, \alpha_2, \beta_2\}|=3$ and $\alpha_1=\beta_2$. This can be calculated in the same way as the previous case.

---\enskip
 $|\{\alpha_1, \beta_1, \alpha_2, \beta_2\}|=3$ and $\beta_1=\beta_2$. This means that $X_{\alpha_1}<X_{\beta_1}$ and $X_{\alpha_2}<X_{\beta_1}$ are both true. Hence
$$\Pr(O_{\alpha_1,\beta_1}\cap O_{\alpha_2,\beta_2})=\frac{1}{\binom{3k}{k}}.$$
As in the previous cases, there are  $d(d-1)(d-2)$ possibilities for this.

---\enskip
 $|\{\alpha_1, \beta_1, \alpha_2, \beta_2\}|=3$ and $\alpha_1=\alpha_2$.
This is the same calculation as the previous case.

Combining the calculations above and using $2\leq d\leq k$ we arrive at the following inequalities.
\begin{multline*}
  \Pr(E_i)\geq \frac{d(d-1)}{\binom{2k}{k}}-\frac{1}{2}\frac{d(d-1)(d-2)(d-3)}{\binom{2k}{k}\binom{2k}{k}}
     -\frac{d(d-1)(d-2)}{\binom{3k}{2k}\binom{2k}{k}}-\frac{d(d-1)(d-2)}{\binom{3k}{k}} \\
  =  \frac{d(d-1)}{\binom{2k}{k}} \left(1 - \frac{1}{2}\frac{(d-2)(d-3)}{\binom{2k}{k}}
     -\frac{(d-2)}{\binom{3k}{2k}}-\frac{(d-2)\binom{2k}{k} }{\binom{3k}{k}}\right)\\
  \geq  \frac{d(d-1)}{\binom{2k}{k}} \left(1 - \frac{(k-2)(k-3)}{ 2\binom{2k}{k}}
     -\frac{(k-2)}{\binom{3k}{2k}}-\frac{(k-2)\binom{2k}{k} }{\binom{3k}{k}}\right):= \frac{d(d-1)}{\binom{2k}{k}}\cdot \frac{1} {1+\gamma(k)}.
\end{multline*}
Here $\gamma(2)=0$, $\gamma(3)= 1/3$, $\gamma(4)<0.44$, $\gamma(5)<0.36$ and then it exponentially converges to 0 as $k\to\infty$.

Summing these lower bounds for all $i$ we get~\eqref{eq221}. \end{proof}

\subsection{Proof of the upper bound for $c(k)$}\label{sec_upper}
We prove~\eqref{eq12} in the following form. For all $\NN, k\geq 2$,
\begin{equation*}
  f_2(\NN,k) \leq  \trivup+ (1+\gamma(k))k\binom{2k}{k},
\end{equation*}
where we define $\gamma(2)=\frac{1}{3}$ and $\gamma(k)$ comes from the proof of Lemma~\ref{cross-int} for $k \geq 3$.

Consider a semi-intersecting family $\m{S}$  with parameters $\NN$ and $k$ and base sets $A$ and $B$.
Since $(1+\gamma(k))k\binom{2k}{k} \geq 2k^3$ we may suppose $|\m{S}|> 2k^3$.


\begin{Lemma}\label{gyenge}
If $|\m{S}|> 2k^3$
then either all degrees in $A$ are at most $k$, or all degrees in $B$ are at most $k$.
\end{Lemma}

\begin{proof}
Assume that there is an $a \in A$  with degree more than $k$. Let $S_1, S_2, \ldots , S_{k+1} \in \m{S}$ be some (distinct) sets containing $a$ and define $X:=  \bigcup_{1\leq i \leq k+1} (S_i \cap A)$.
Note that $|X| \leq k^2$ as we take the union of $k+1$ sets with $k$ elements, all containing $a$.
The sets $B\cap S_i$ are pairwise disjoint for $1\leq i \leq k+1$ so no $T \in \m{S}$ can intersect each of them in $B$.
Hence $T \cap X \neq \emptyset$ for all $T \in \m{S}$.

We claim that the degree of each $y\in B$ is at most $|X|$, i.e., $\deg_{\m{S}}(y)\leq k^2$. Indeed, a set $T\in \m{S}$ with $y\in T$ contains a pair $\{ x, y\}$ with $x\in X$ and every such pair can appear in at most one member of $\m{S}$.

Similarly, $b\in B$ and $\deg_{\m{S}}(b)>k$ imply $\deg_{\m{S}}(x)\leq k^2$ for every $x\in A$.

Fix any member $T\in \m{S}$. Since $\m{S}$ is intersecting we obtain
$|\m{S}|\leq \sum_{z\in T}\deg_{\m{S}}(z) \leq 2k\cdot k^2$, and we are done.
\end{proof}

From now on, we may suppose that  $\deg_{\m{S}}(y)\leq k$ for all $y \in B$.
Starting with $\m{S}_0:=\m{S}$ we define a series of families $\m{S}_0 \supset \m{S}_1 \supset \ldots \supset \m{S}_q$ as follows. 
If the families $\m{S}_0, \m{S}_1, \ldots , \m{S}_{i-1}$ have already been created, and the members of $\m{S}_{i-1}$ were pairwise disjoint in $B$ then we let $q:=i-1$ and $\m{S}_{q}:=\m{S}_{i-1}$.  Note that, $\m{S}_{q} \leq\left\lfloor\frac{n}{k}\right\rfloor$.

Otherwise, define $\m{S}_i$ as follows.
Take a vertex $p_i \in B$ with maximum degree in $\m{S}_{i-1}$, let $\mathcal{Z}_i \subset \m{S}_{i-1}$ be the family of sets containing $p_i$, and let $d_i:=|\m{Z}_i|$. We have $2 \leq d_i \leq k$.
Denote by $\mathcal{M}_{i} \subset \m{S}_{i-1} \setminus \m{Z}_i$ the family of sets which intersect at least one set of $\mathcal{Z}_{i}$ in $B$. Finally, let $\m{S}_i:=\m{S}_{i-1} \setminus (\m{Z}_i \cup \m{M}_i)$.

Now we give an upper bound for $|\m{S}_{i-1} \setminus \m{S}_{i}|$.
Since $p_i\in \cap \m{Z}_i$ we have
$\left|B \cap \bigcup_{Z \in \m{Z}_i} Z \right| \leq 1+(k-1)d_i$.
An element of $B \cap \left( \bigcup \m{Z}_i\right) \setminus \{ p_i\}$ can meet at most $d_i-1$
 members of $\mathcal{M}_{i}$, so we get $|\m{M}_i| \leq (k-1)d_i(d_i-1)$.
Hence $|\m{S}_{i-1} \setminus \m{S}_{i}|=|\m{Z}_i \cup \m{M}_i| \leq d_i+(k-1)d_i(d_i-1)$ and
$$|\m{S}|=|\m{S}_q|+\sum_{1\leq i\leq q} |\m{S}_{i-1} \setminus \m{S}_i| \leq \trivup+\sum_{1\leq i\leq q} (d_i+(k-1)d_i(d_i-1)).
$$
We get
$$  |\m{S}|-\trivup \leq \sum_{1\leq i\leq q} d_i+(k-1)\sum_{1\leq i\leq q} d_i(d_i-1) \leq k \sum_{1\leq i\leq q} d_i(d_i-1).
$$
We need to bound $\sum_{1\leq i\leq q}  d_i(d_i-1)$.

Observe that  the sets in the family $\mathcal{Z}_{i}$ are pairwise disjoint in $A$ as they have a common element in $B$.
Define $\cA_i$ as $\{ Z\cap A: Z\in \m{Z}_i \}$.
If $S \in \m{S}_{i}$ and $Z \in \mathcal{Z}_{i}$ then $S\cap Z\cap B=\emptyset$, as any set from $\m{S}_{i-1}$ that intersects $Z$ in $B$ is in $\m{Z}_i \cup \m{M}_i$ by definition. Hence $S \cap Z\cap A\neq \emptyset$. In particular, any $Z \in \m{Z}_i$ and $Z' \in \m{Z}_j$ intersect in $A$ if $i < j$, i.e., $X\cap X'\neq \emptyset$ if $X\in \cA_i$, $X'\in \cA_j$, and $i \neq j$.

If $q\geq 2$ then $\cA_1, \dots, \cA_{q} $ form a $k$-uniform cross intersecting matching and then Lemma~\ref{cross-int} completes the proof. In case of $q\leq 1$ we have $|\m{S}| \leq \trivup + 2k^2(k-1)$ and we are done. \qed

\section{Multiproducts of the complete graphs} \label{sec_k1}

\subsection{Algebraic upper bound for the product of complete graphs}\label{ss:completes_upper}

Complete graphs are also Kneser graphs with $k=1$.
We prove the upper bound  $f_\ell(\NN,1)\leq \NN \ell-\ell+1$ in Theorem~\ref{th_1} in the following stronger form.
Suppose that $\ell\geq 2$, $n_1, \dots, n_\ell$ are positive integers and $A_1, \dots, A_\ell$ are disjoint sets
 of sizes $n_1, \dots, n_\ell$, $V:=A_1 \cup A_2 \cup \ldots \cup A_\ell $.

\begin{Th}\label{th_1_upper}
Let $G$ be the xor-product of the complete graphs $K_{n_1}, \dots, K_{n_\ell}$.
Then $\omega(G)\leq |V|-\ell+1$.
\end{Th}

The vertices of $G$ corresponds to $\ell$-element sets $T$ with $|T\cap A_i|=1$ for each $i$.
A clique in $G$ corresponds to an $\ell$-semi-intersecting family $\cS$ of $\ell$-element subsets of $V$, i.e.,
 for distinct $S,T\in \cS$ we have $|S\setminus T|= \ell-|S\cap T|$ is odd, so $(\ell+1+ |S\cap T|)$ is even.
Let $\bbF_2$ be the 2-element field.
For every subset $X\subseteq V$ let $\vec X \in \bbF^V$ denote the characteristic vector $X$. Thus $\vec \emptyset $ is the $|V|$ dimensional zero-vector.
Let $\cA$ denote the family $\{ A_1, \dots, A_\ell \}$.

\begin{Lemma} \label{lemma_1}
Suppose that the cardinality $|\cS|$ is odd and
 suppose that the vectors $\{ \vec{X} : X\in \cS\cup \cA \}$
 have a non-trivial linear dependency, $\sum_{X\in \cS\cup \cA} \alpha (X)\vec{X} = \vec{\emptyset}$ for some
 $\alpha(X) \in \bbF_2$, not all coefficients are zero. Then this dependency is unique and $\alpha(X)=1$ for each $X\in \cS\cup \cA$.
\end{Lemma}

\begin{proof}[Proof of Lemma~\ref{lemma_1}]
The scalar product $\langle \vec{X}, \vec{Y} \rangle = |X\cap Y|$.
So for any $Y\subseteq V$,
$$ 0= \langle \vec{\emptyset}, \vec{Y} \rangle = \sum _{X\in \cS\cup \cA} \alpha (X)|X\cap Y| .
$$
Here every integer is taken modulo 2.
Substituting to $Y$ a single element $v\in A_i$, then a fixed member $A_i\in \cA$, and finally a member $T\in \cS$ we get
\begin{eqnarray}
0 &=& \sum_{S: v\in S\in \cS} \alpha(S) + \alpha(A_i), \label{eq1_v}\\
0 &=& \left( \sum_{S\in \cS} \alpha(S) \right) +\alpha(A_i)|A_i|,\label{eq1_Ai}\\
0 &=& \left( \sum_{S\in \cS} |S\cap T|\alpha(S) \right) +\left( \sum _j \alpha(A_j)\right).  \label{eq1_T}
\end{eqnarray}
Add $(\ell+1)$ times \eqref{eq1_Ai} to \eqref{eq1_T}. For given $T$ and $A_i$ we get
$$
 0 = \left( \sum_{S\in \cS} (\ell+1+ |S\cap T|)\alpha(S) \right)+ (\ell+1)\alpha(A_i)|A_i| +\left( \sum _j \alpha(A_j)\right).
  $$
For distinct $S,T\in \cS$  $(\ell+1+ |S\cap T|)$ is even and for $S=T$ we have $\ell+1+ |S\cap T|=2\ell+1=1$ (in $\bbF_2$).
So the first term in the last displayed formula is exactly $\alpha(T)$. The second and the third terms are independent from $T$, so we obtain that all $\alpha(T)$ are equal.

If $\alpha(T)=0$ for each $T\in \cS$ then \eqref{eq1_v} gives that $\alpha(A_i)=0$ for all $i$, a contradiction.
Therefore each $\alpha(T)=1$, so the first term in \eqref{eq1_Ai} is $|\cS|$. By our assumptions this is odd, so $\alpha(A_i)|A_i|$ should be odd. In particular $\alpha(A_i)=1$ for all $i$.
\end{proof}

\begin{proof}[Proof of Theorem~\ref{th_1_upper}]
If $|\cS|$ is odd, then by Lemma~\ref{lemma_1}  the vectors $\{ \vec{X}: X\in \cS\cup \cA\}$ are either linearly independent in $\bbF_2^V$ or has a unique linear dependency. So they generate a subspace of dimension at least $|\cS|+|\cA|-1$. This is at most $|V|$ and we are done.

If $|\cS|$ is even then we can assume that $|\cS|\geq 2$. Take two distinct members $T_1, T_2\in \cS$.
Then $|\cS\setminus \{ T_i\}|$ is odd (for $i=1,2$). If either of the set of vectors $\{ \vec{X}: X\in (\cS\setminus \{ T_i\})\cup \cA\}$ is independent, we get $|\cS|+|\cA|-1\leq |V|$ as desired.
If both are dependent, then again by Lemma~\ref{lemma_1} they have unique linear dependencies, namely
  $\sum \{ \vec{X}: X\in (\cS\setminus \{ T_i\})\cup \cA\}=\vec{\emptyset}$.
Adding up these equations we get $\vec{T_1}+\vec{T_2}= \vec{\emptyset}$. This contradiction completes the proof.
\end{proof}

\subsection{An explicit construction for the case of complete graphs}\label{sscompletes_lower}

We prove the lower bound  $f_\ell(\NN,1)\geq \ell \NN -2\ell-1$ in Theorem~\ref{th_1} in the following stronger form.
Suppose that $\ell\geq 3$, $A_1, \dots, A_\ell$ are disjoint sets
 of sizes $n_1, \dots, n_\ell$, $V:=A_1 \cup A_2 \cup \ldots \cup A_\ell $.

\begin{Th}\label{th_1_lower}
Let $G$ be the xor-product of the complete graphs $K_{n_1}, \dots, K_{n_\ell}$.
Suppose that $\ell\geq 3$ and $n_i\geq 2$ for each $i\in [\ell]$.
Then $\omega(G)\geq |V|-2\ell-1$.
\end{Th}

\begin{proof}
We show a construction.
For a given partition $A_1, \dots, A_\ell$ ($\ell\geq 3$) we call the family of sets $\cB:=\{B_1, \dots, B_\ell\}$ an $\ell$-{\em core} if
\begin{itemize}
    \item[(i)] each $B_i$ is an $(\ell-1)$-set with $B_i\cap A_i=\emptyset$ but $|B_i\cap A_j|=1$ for $i\neq j$ and
    \item[(ii)] $|B_i\cap B_j|\not \equiv \ell \pmod 2$ for all $1\leq i,j\leq \ell$.
\end{itemize}
This intersection condition can be reformulated as $|B_i\cap B_j|+\ell$ is always odd.
Let $U(\cB)$ denote $\cup B_i$.
A core $\cB$ generates an $\ell$-uniform family $\cS(\cB)$ by enlarging each core set by extra elements as follows.
$$
   \cS(\cB):= \{ B_i\cup \{ x\}: i\in [\ell], \, x\in A_i\setminus U\}.
  $$
We claim that $\cS(\cB)$ is an $\ell$-semi-intersecting family with $k=1$ of size $|V|-|U|$.
Indeed, by definition each $S \in \cS(\cB)$ intersects every $A_i$ in exactly one element.
Let $S, T \in \cS(\cB)$ with $S \neq T$.
Say $S=B_i\cup \{ x\}$ and $T=B_j\cup \{y\}$.
Then $S\cap T= B_i\cap B_j$.
So $S\cap T\cap A_\alpha=\emptyset$ in $\ell-|B_i\cap B_j|$ cases of $\alpha$.
Here $\ell-|B_i\cap B_j|$ is odd by (ii), so $\cS(\cB)$ is an $\ell$-semi-intersecting family.

The next step in the proof of Theorem~\ref{th_1_lower} is to find a core $\cB$ with small $|U(\cB)|$.
We need a couple of more definitions.
The {\em type} of $\cB$ is the multiset $\{ |U\cap A_i|: i\in [\ell] \}$.
The set $A_i\cap U$ is called the $i$th {\em class} of $\cB$.

\begin{Lemma}\label{le:fusion}
Suppose that $p,q\geq 3$, $\cB_p'$
 is a $p$-core of type $(x_1, \dots, x_p)$ and $\cB_q''$
 is a $q$-core of type $(y_1, \dots, y_q)$. Then there exists a $(p+q-1)$-core $\cB$
 of type $(x_1, \dots, x_{p-1}, x_p+y_1, y_2, \dots, y_q)$.
\end{Lemma}

\begin{proof}[Proof of Lemma~\ref{le:fusion}]
Suppose that $U(\cB_p')$ and  $U(\cB_q'')$ are disjoint.
We have  $|\cB_p'|=p$, $|\cB_q''|=q$.
The classes of $\cB_p'$ are denoted by
 $A_1', \dots, A_p'$ ($|A_i'|=x_i$), its hyperedges are $B_1', \dots, B_p'$.
The classes of $\cB_q''$ are denoted by
 $A_1'', \dots, A_q''$ ($|A_j''|=y_j$), its hyperedges are $B_1'', \dots, B_q''$.
We define the core $\cB=\cB_{p+q-1}$ as follows.
Its classes are $A_1, \dots, A_{p+q-1}$ where $A_i:= A_i'$ for $1\leq i\leq p-1$, $A_p:=A_p'\cup A_1''$, and
 $A_j:= A_{j-p+1}''$ for $p+1\leq j \leq p+q-1$.
The hyperedges $B_1, \dots, B_{p+q-1}$ of $\cB_{p+q-1}$ are defined as unions of the form $B_\alpha'\cup B_\beta''$ as follows:
$B_p:= B_p'\cup B_1''$ and in general $B_i:= B_i'\cup B_1''$ for $1 \leq i\leq p$ and $B_j:= B_p'\cup B_{j-p+1}''$ for $p+1\leq j\leq p+q-1$.

We claim that $\cB$ is a $(p+q-1)$-core.
 $B_\alpha\cap A_\alpha=\emptyset$ and $|B_\alpha\cap A_\beta|=1$ for $\alpha\neq \beta$ follows from the definition of $\cB$.
Consider $|B_\alpha \cap B_\beta|$. We have to show that $|B_\alpha \cap B_\beta| +(p+q-1)$ is odd.
Write $B_\alpha$ in the form $B_e'\cup B_f''$ and let $B_\beta=B_g'\cup B_h''$ where $B_e', B_g'\in \cB'$ and $B_f'', B_h''\in \cB''$.
We have $B_\alpha \cap B_\beta= (B_e'\cup B_f'') \cap (B_g'\cup B_h'')$ which is the disjoint union of
$B_e'\cap B_g'$ and $B_f''\cap B_h''$.
Since $|B_e'\cap B_g'|+p$ and $|B_f''\cap B_h''|+q$ are both odd their sum is even, so
$$
 |B_\alpha \cap B_\beta|+p+q-1= |B_e'\cap B_g'|+|B_f''\cap B_h''|+p+q-1$$
   is odd.
So $\cB$ is a $(p+q-1)$-core, completing the proof of~Lemma~\ref{le:fusion}.
\end{proof}

The procedure described in the proof of Lemma~\ref{le:fusion}
will be referred to as the {\em fusion} of $\cB_p$ and $\cB_q$. Using this construction, we prove by induction that there exists an $\ell$-core $\cB$ with $|U(\cB)| \leq 2\ell+1$ if $\max_{i} n_i\geq 3$. Note that we can assume that $\max_{i} n_i\geq 3$ as if each $n_i=2$ then $|V|= 2\ell$, so the lower bound from Theorem~\ref{th_1_lower} obviously holds.
First, we define an $\ell$-core for $\ell=3,4,5$.

There is a $3$-core $\cB_3$ of type $(2,2,2)$ with three sets $B_\alpha:=\{ a_{\alpha-1, \alpha}, a_{\alpha+1, \alpha}\}$ (indices are taken modulo 3) where these $a$'s are six distinct vertices with $a_{i,j}\in A_i$ ($i,j\in \{ 1,2,3\}$, $i\neq j$).

There is a $4$-core $\cB_4$ of type $(2,2,2,1)$ on $7$ vertices $\{ a_{1,2}, a_{1,3}, a_{2,1}, a_{2,3}, a_{3,1}, a_{3,2}, a_4\}$ where
  $\{ a_{1,2}, a_{1,3}\}\subset A_1$,  $\{ a_{2,1}, a_{2,3}\}\subset A_2$, $\{a_{3,1}, a_{3,2}\}\subset A_3$, and $a_4\in A_4$.
The core sets are  $B_1:=\{ a_{2,1}, a_{3,1}, a_4\}$, $B_2:=\{ a_{1,2}, a_{3,2}, a_4\}$, $B_3:=\{ a_{1,3}, a_{2,3}, a_4\}$, and $B_4:=\{ a_{1,3}, a_{2,1}, a_{3,2}\}$.

There is a $5$-core $\cB_5$ of type $(2,2,2,1,1)$ on $8$ vertices $\{ a_{1,2}, a_{1,3}, a_{2,1}, a_{2,3}, a_{3,1}, a_{3,2}, a_4, a_5\}$ where
  $\{ a_{1,2}, a_{1,3}\}\subset A_1$,  $\{ a_{2,1}, a_{2,3}\}\subset A_2$, $\{a_{3,1}, a_{3,2}\}\subset A_3$, $a_4\in A_4$, and $a_5\in A_5$.
The core sets are  $B_1:=\{ a_{2,1}, a_{3,1}, a_4, a_5\}$, $B_2:=\{ a_{1,2}, a_{3,2}, a_4, a_5\}$, $B_3:=\{ a_{1,3}, a_{2,3}, a_4, a_5\}$, $B_4:=\{ a_{1,3}, a_{2,1}, a_{3,2}, a_5\}$, and $B_5:=\{ a_{1,2}, a_{2,3}, a_{3,1}, a_4\}$.

Now we give the construction for any $\ell \geq 3$. Fusing $m$ copies of $\cB_5$ of type $(1,2,2,2,1)$ we obtain a $(4m+1)$-core $\cB_{4m+1}$ of type $(1, 2, \dots, 2,1)$ for each $m\geq 1$.
The fusion of a $\cB_4$ of type $(2,2,2,1)$ and a $\cB_{4m+1}$ defines a $(4m+4)$-core
  $\cB_{4m+4}$ of type $(2,2\dots, 2, 1)$ for each $m\geq 0$.
Fusing this with a $\cB_4$ of type $(1,2,2,2)$ yields a $(4m+7)$-core of type $(2,\dots, 2)$ (for each $m\geq 0$).
Finally, fusing $\cB_3$ of type $(2,2,2)$ and a  ${4m+4}$-core of type $(1,2, \dots, 2)$ ($m\geq 0$) one gets a
 $(4m+6)$-core of type $(2,2,3,2,\dots, 2)$, which finishes the proof.
\end{proof}

\subsection{Remarks}\label{ss_ell=1_remarks}

Note that in the proof of Theorem~\ref{th_1_upper} we have adapted a method of Deza, Frankl and Singhi~\cite{deza1983functions}.
Their `even town theorem' can be applied to prove $f_\ell (\NN,1) \leq \ell \NN-\ell+1$ when $\ell $ and $\NN$ are both even.
If $\ell $ is even and $\NN$ is odd one can still apply the even town theorem to get $\ell \NN+1$, the same upper bound  as in Alon and Lubetzky~\cite{xorproduct}.
In the case $\ell $ is odd the bound $f_\ell (\NN,1) \leq \ell \NN$ follows by a theorem of Frankl and Wilson~\cite{frankl1981intersection}.

Suppose that there exists a finite projective plane of order $\ell-1$, i.e., an $\ell$-uniform, $\ell$-regular set system $\cL$ of size $\ell^2-\ell+1$ such that $|L\cap L'|=1$ for all pairwise intersections. Also suppose that $\ell$ is even. Take any vertex $v$, we have $\ell$ lines containing it, $L_1, \dots,  L_{\ell}$.
Set  $A_i:=L_i\setminus \{ v\}$, and let $\cS:= \cP \setminus \{ L_1, \dots, L_\ell\}$.
This $\cS$ is an $\ell$-semi-intersecting family of size $(\ell-1)^2$ on $\ell \times (\ell-1)$ vertices.
Since such finite planes exist whenever $\ell-1$ is a power of an odd prime, we got
 infinitely many cases when the lower bound is tight in Theorem~\ref{th_1_upper}. This example was also mentioned by Alon and Lubetzky~\cite{xorproduct, AL2}. They were more interested from coding theory point of view, i.e., when $n$ is fixed and $\ell\to \infty$.

Finding the exact value of $f_\ell (\NN,1)$ is still open.

\section{Higher powers, the general case} \label{sec_generalpowers}

In this section we study the order of magnitude of $f_\ell(\NN,k)$.

\subsection{Proof of the upper bound~\eqref{eq13} by induction on $\ell$}\label{ss41}

Suppose that $\ell\geq 2$ and let $\cS$ be
an $\ell$-semi-intersecting family with parameters $\NN$ and $k$ and base sets $A_1, \dots, A_\ell$. Take any $v \in A_i$. Define
    $\m{S}[v]:=\{S\setminus A_i:  v\in S\in \m{S} \}$.
Then $\m{S}[v]$ does not contain multiple hyperedges, it is an $(\ell-1)$-semi-intersecting family.
Hence $|\m{S}[v]|= \deg_{\cS}(v) \leq f_{\ell-1}(\NN,k)$.
Take this inequality for each $v\in A_i$ and suppose that $|\cS|$ has maximum size. We get
\begin{equation}\label{eq411}
  f_\ell(\NN,k)=|\m{S}| = \frac{1}{k} \sum_{v\in A_i} \deg_{\m{S}}(v)   \leq   \frac{\NN}{k}  f_{\ell-1}(\NN,k).
\end{equation}

If $\ell$ is even then $\cS$ is intersecting. Taking the degrees of any given $T\in \m{S}$ we obtain
\begin{equation}\label{eq412}
  f_\ell(\NN,k)=|\m{S}| \leq 1 +\sum_{v\in T}(\deg_\m{S}(v)-1) \leq  1+k\ell\left(f_{\ell-1}(\NN,k)-1\right)\leq k\ell f_{\ell-1}(\NN,k).
\end{equation}

We have $f_1(\NN,k)\leq \NN/k$.
Apply~\eqref{eq412}, we get $ f_2(\NN,k) \leq 2k f_{1}(\NN,k)\leq 2\NN $.
Then~\eqref{eq411} gives $ f_3(\NN,k) \leq (\NN/k) f_{2}(\NN,k)\leq 2\NN^2/k $.
Apply again~\eqref{eq412}, we get $ f_4(\NN,k) \leq 4k f_{3}(\NN,k)\leq 2\cdot 4 \cdot \NN^2$.
Continuing this way, we get for each even $\ell$ that $ f_\ell(\NN,k) \leq  2\cdot4\cdot \dots \cdot \ell \cdot \NN^{\ell/2}$ and
 $ f_\ell(\NN,k) \leq  2\cdot4\cdot \dots \cdot (\ell-1) \cdot \NN^{(\ell+1)/2}/k$
when  $\ell$ is odd.
\qed

\subsection{Construction showing the lower bound~\eqref{eq14}}\label{ss42}

 Note that $f_\ell(\NN,k)$ is monotonous in $\NN$ and also increases monotonously in $\ell$, since an $\ell$-semi-intersecting family $\m{S}$ can be extended to an $(\ell+1)$-semi-intersecting family by adding $A_{\ell+1}$ to the base sets and a fixed $k$-element $S_{\ell+1} \subset A_{\ell+1}$ to all $S \in \m{S}$.
So it is enough to prove the theorem for $\ell=2^t-1$ where $t\geq 2$ is an integer, and we also suppose that $k \geq t$.
Let $m:=\left\lfloor \frac{\NN}{k} \right\rfloor$.

Take $\ell$ disjoint sets $A_1, \dots, A_\ell$ of sizes $|A_\alpha|=mk$.
We are going to define an $\ell$-semi-intersecting family $\m{S}$ of size $m^t$ with parameters $mk$ and $k$ with these base sets.
Let $H$ be $0$-$1$ matrix of size $\ell \times t$ with $2^t-1$ pairwise distinct nonzero rows. Note that this matrix is unique up to a permutation of its rows.
Let $C$ be an  $\ell \times t$ matrix with non-negative integer entries such that $C_{\alpha,\beta}=0$ if and only if $H_{\alpha,\beta}=0$
 and the row sums are $k$, i.e., $\sum_{1\leq \beta\leq t} C_{\alpha,\beta}=k$. This is possible, as $k \geq t$.
Partition each $A_\alpha$ into subsets $A_{\alpha,\beta}^p$ where $1\leq p\leq m$ and $|A_{\alpha,\beta}^p|= C_{\alpha,\beta}$. In particular, let $A_{\alpha,\beta}^p=\emptyset$
if $C_{\alpha,\beta}=0$.
Make another partition of $\cup A_\alpha$ into $mt$ sets by joining some of the $A_{\alpha,\beta}^p$ as follows. For each $\beta$  and $p$ where $1\leq \beta \leq t$ and $1\leq p\leq m$ define
$$ S_\beta^p:= \bigcup_{1\leq \alpha \leq \ell}  A_{\alpha,\beta}^p.
$$
For each of the $m^t$ functions $\varphi: [t]\to [m]$  define
    $S_{\varphi}:=\cup_{1\leq \beta\leq t}  S_\beta^{\varphi(\beta)}$. Finally, let $\m{S}:=\{S_{\varphi}: \varphi: [t]\to [m] \}$.

Let us prove that $\m{S}$ is an $\ell$-semi-intersecting family.
Each $\S_\varphi\in \m{S}$ intersects every $A_\alpha$ in $k$ elements, as
    $$
    |S_\varphi \cap A_\alpha|=\left|\bigcup_{1\leq \beta\leq t} (A_\alpha \cap S_\beta^{\varphi(\beta)})\right|=
        \left| \bigcup_{1\leq \beta\leq t}  A_{\alpha,\beta}^{\varphi(\beta)} \right| = \sum_{1\leq \beta\leq t}   C_{\alpha,\beta}=k.
        $$

Let $D_\beta \subset [\ell]$ denote the set of indices of nonzero elements of the column $\beta$ of $H$, i.e.,
  $D_\beta:= \{ \alpha: H_{\alpha,\beta}=1\}$.
For any set $X\subset \cup A_\alpha$ let $\pi(X)\subset [\ell]$ denote its projection,  $\pi(X):=\{ \alpha\in  [\ell]:  A_\alpha\cap X\neq \emptyset\}$.
We have $\pi(S_\beta^p)=D_\beta$ for all $p$. Even more, each such set has the type $(C_{1, \beta}, \dots, C_{\ell, \beta})$, i.e., $|S_\beta^p\cap A_\alpha|$ is exactly $C_{\alpha,\beta}$.
Note that for any $Q\subseteq [t]$ we have $|\cup_{q\in Q} D_q|= 2^t-2^{t-q}$, an even number except in the case $Q=[t]$.

    Given two functions $\varphi,\sigma$ we claim that $|\pi(S_\varphi \cap S_\sigma)|$ is even except in the case $\varphi=\sigma$. This implies that
$\m{S}$ is $\ell$-semi-intersecting, as claimed.
We have
$S_\varphi \cap S_\sigma= \left( \cup S_\beta^{\varphi(\beta)} \right) \cap \left(  \cup S_\beta^{\sigma(\beta)} \right)$.
Since the sets $S_\beta^{p}$ form a partition of $\cup A_\alpha$ we have that $S_\varphi \cap S_\sigma= \cup \{ S_\beta^{p}:
  \varphi(\beta)= \sigma(\beta)\} $.
Hence $\pi(S_\varphi \cap S_\sigma)= \cup \{ D_\beta:
  \varphi(\beta)= \sigma(\beta)\}$. This set has even cardinality whenever $\varphi \neq \sigma$.  So
   we find that $S_\varphi$ and $S_\sigma$ are disjoint in an odd number of base sets $A_\alpha$ which finishes the proof.
\qed

\bibliography{biblio25}
\bibliographystyle{plain}

\bigskip

\ \\
\noindent
{\bf Zolt\'an F\"uredi}\\
Alfr\'ed R\'enyi Institute for Mathematics \\
Re\'altanoda u. 13--15, H-1364 Budapest, Hungary\\
\texttt{furedi\svgperiod zoltan[at]renyi\svgperiod hu}
\medskip

\ \\
\noindent
{\bf András Imolay}\\
E\"otv\"os Lor\'and University \\
H-1117 Budapest, Pázmány Péter sétány 1/C \\
\texttt{andras\svgperiod imolay[at]ttk\svgperiod elte\svgperiod hu}
\medskip

\ \\
{\bf Ádám Schweitzer}\\
KTH Royal Institute of Technology\\
Kungliga Tekniska högskolan, 100 44 Stockholm\\
\texttt{adasch[at]kth\svgperiod se}

\end{document}